\documentclass{article}

\usepackage[authoryear]{natbib}
\usepackage[utf8]{inputenc}
\usepackage[english]{babel}
\usepackage{amsmath}
\usepackage{amsthm}
\usepackage{amssymb}
\usepackage{clrscode3e}
\usepackage[hidelinks]{hyperref}
\usepackage{float}
\usepackage{xspace}
\usepackage{xcolor}
\usepackage{soul}

\usepackage[T1]{fontenc}

\usepackage{todonotes}
\usepackage[a4paper,left=2.5cm,right=2.0cm,top=2.5cm, bottom=2.0cm]{geometry}
\usepackage{pdflscape}
\usepackage{afterpage}
\usepackage{rotating}
\usepackage{tikz}

\usepackage{pgfplots}
\pgfplotsset{compat=newest}

\usepackage{mathtools}

\usepackage{booktabs}

\newcommand{\eps}{\varepsilon}

\theoremstyle{definition}
\newtheorem{lemma}{Lemma}
\newtheorem{theorem}{Theorem}
\newtheorem{corollary}{Corollary}
\newtheorem{definition}{Definition}

\newcommand{\PP}{\mbox{P}\xspace}
\newcommand{\NP}{\mbox{NP}\xspace}
\newcommand{\APX}{\mbox{APX}\xspace}

\DeclareMathOperator{\opt}{\mbox{OPT}}

\DeclarePairedDelimiter{\ceil}{\lceil}{\rceil}

\newcommand{\patternset}{\mathcal{P}}
\newcommand{\calb}{\mathcal{B}}
\newcommand{\calf}{\mathcal{F}}

\newcommand{\cali}{\mathcal{I}}

\newcommand{\calk}{\mathcal{K}}
\newcommand{\call}{\mathcal{L}}

\newcommand{\calp}{\mathcal{P}}
\newcommand{\calr}{\mathcal{R}}
\newcommand{\cals}{\mathcal{S}}
\newcommand{\calt}{\mathcal{T}}

\newcommand{\val}{\ensuremath{\textit{V}}}

\newcommand{\vvp}{\ensuremath{\val_{\textit{VBPP}}}}
\newcommand{\vbps}{\ensuremath{\val_{\textit{BPPS}}}}

\newcommand{\footremember}[2]{%
    \footnote{#2}
    \newcounter{#1}
    \setcounter{#1}{\value{footnote}}%
}
\newcommand{\footrecall}[1]{%
    \footnotemark[\value{#1}]%
} 

\begin{document}

\title{Algorithms for the Bin Packing Problem with Scenarios}
\author{
  Yulle G. F. Borges\footnote{Corresponding Author, Email Address: \texttt{glebbyo@ic.unicamp.br}} \footremember{ic}{Institute of Computing, University of Campinas, Campinas--SP, Brazil.} \and  
    Vin\'icius L. de Lima\footrecall{ic} \and
    Fl\'avio K. Miyazawa\footrecall{ic} \and
    Lehilton L. C. Pedrosa\footrecall{ic}\and
    Thiago A. de Queiroz\footremember{ufcat}{Institute of Mathematics and Technology, Federal University of Catal\~ao, Catal\~ao--GO, Brazil.}\and
    Rafael C. S. Schouery\footrecall{ic}
}


\date{}

\maketitle

\begin{abstract}
This paper presents theoretical and practical results for the bin packing problem with scenarios, a generalization of the classical bin packing problem which considers the presence of uncertain scenarios, of which only one is realized. For this problem, we propose an absolute approximation algorithm whose ratio is bounded by the square root of the number of scenarios times the approximation ratio for an algorithm for the vector bin packing problem. We also show how an asymptotic polynomial-time approximation scheme is derived when the number of scenarios is constant. As a practical study of the problem, we present a branch-and-price algorithm to solve an exponential model and a variable neighborhood search heuristic. To speed up the convergence of the exact algorithm, we also consider lower bounds based on dual feasible functions. Results of these algorithms show the competence of the branch-and-price in obtaining optimal solutions for about 59\% of the instances considered, while the combined heuristic and branch-and-price optimally solved 62\% of the instances considered.

\end{abstract}

%

\section{Introduction}
Cutting and packing problems have been widely studied in the context of Operations Research, mainly because of their properties and real-world applicability~\citep{WascherHS07}.  Among these problems, the \emph{Bin Packing Problem} (BPP) asks to pack a set of one-dimensional items, each of a given size, into the least possible number of bins of given identical capacities, where the total size of items in a bin does not exceed the bin capacity.  Several variants of this problem have been studied in the literature, either due to their theoretical interest, or their high applicability in different industries~\citep{sweeneyP92,DAHMANI2013}.

The BPP and its variants have been extensively investigated since the thirties~\citep{Kantorovich60}. Such problems are among the most studied in approximation contexts. A recent survey by~\citet{coffmanCGMV13} presented over 180 references related to approximation results for the BPP and its variants. 
In particular, the BPP is shown to be \APX-hard, and indeed no algorithm has an approximation factor smaller than $3/2$, unless $\PP = \NP$~\citep{Vazirani01}. 
Regarding practical techniques for the BPP, a survey by~\citet{DELORME2016} reviewed models and solution methods, and experimentally compared the available software tools to solve the problem. Recent practical contributions for the BPP include the generalized arc-flow formulation by~\citet{BRANDAO2016}, the cooperative parallel genetic algorithm by~\citet{KUCUKYILMAZ2018}, the reflect formulation by~\citet{Delorme19}, the branch-and-cut-and-price algorithm by~\citet{WLBL20} and the framework of~\citet{Lima2023}.

Several variants of the BPP have also been extensively investigated in the literature. For example, in the variant with fragile objects, the capacity of a bin is directly influenced by its most fragile item~\citep{CLAUTIAUX2014}. In the BPP with precedence constraints, a precedent item cannot be packed in a bin later than its successors~\citep{PEREIRA2016}. The variant with overlapping items assumes that some subsets of items, when packed in the same bin, occupy less capacity than the sum of their individual size~\citep{GRANGE2018}. In the BPP with item fragmentation, items can be split and fragmentally packed~\citep{BERTAZZI2019}. In the generalized BPP, bins may have different cost and capacity and items are associated to a profit and are divided by compulsory (i.e., mandatory to load) and non-compulsory. The aim is to minimize the overall cost based on the bins cost and items profit~\citep{BALDI2019}. In the online BPP, items arrive sequentially and each must be packed before the arrival of the next one, with no information about the next items~\citep{BALOGH19}. The temporal variant considers one additional dimension in the BPP, and the bin capacity should not be violated at any unit of a discretized time horizon~\citep{DELLAMICO2020}.

One of the current challenges to solving practical logistic problems is to deal with the uncertainty arising from real-world applications~\citep{guillainPL21,xuLLJ16,juanKCF18,baghalianRF13}. In particular, one popular way to deal with this issue is by describing scenarios. A scenario is defined as a possible outcome that may arise depending on the problem's uncertain variables. Some authors use scenarios to consider the problem description as combinatorial rather than stochastic, as, e.g.,~\citet{feuersteinMSSSSZ14}. 
Considering the scarcity of works that apply this concept of scenarios to deal with uncertainty in bin packing problems and the existence of practical applications,~\citet{bodisB18} recently introduced the \emph{Bin Packing Problem with Scenarios} (BPPS).

The BPPS is a generalization of the BPP where each item belongs to a subset of scenarios, and a packing must respect the capacity constraints of the bins for each scenario individually. Whilst the set of scenarios is known in advance, only one of the scenarios will be realized. This introduces the possibility of packing in a single bin items whose combined sizes surpasses the bin capacity, as long as the bin capacity is respected in each individual scenario. Although~\citet{bodisB18} introduced three objective functions for the BPPS, we are concerned with the objective of minimizing the number of bins of the worst-case scenario, which is the most challenging one. In this way, the BPP is the particular case of the BPPS where there is a single scenario.
%

A generalization of the BPP that is similar to the BPPS is the \emph{Vector Bin Packing Problem} (VBPP), where bins and items have multiple dimensions and the bin capacity must be respected in all of its dimensions. The objective is the same as in the BPP, to minimize the number of bins. \citet{CAPRARA2001} presented lower bounds, heuristics, and a branch-and-bound (B\&B) algorithm for the VBPP with two dimensions. \citet{ALVES2014} investigated dual feasible functions for this problem. \citet{BULJUBASIC2016} proposed a local search algorithms where a Tabu search and descent search with add and drop moves were used to explore the search space.
In~\citet{HU2017}, a set-covering model is solved by a branch-and-price (B\&P) algorithm. In~\citet{HELER2018}, there is a stabilized branch-and-price algorithm with dual optimal cuts.
Recently,~\citet{WEI2020} developed a branch-and-price algorithm for the vector packing with two dimensions, where a goal cut approach is used to obtain lower bounds and a dynamic programming with branch-and-bound solves the pricing problem, improving previous results of the literature. 

To the best of our knowledge, the recent work by~\citet{bodisB18} is the only one in the literature concerned with the BPPS\@. Motivated by its interesting theoretical aspects, its general applicability and its relation to well-studied problems, this paper proposes theoretical and practical results for the BPPS\@. The contributions of the current work are the following: (i) an absolute approximation algorithm for the BPPS based on the VBPP\@; (ii) an \emph{asymptotic polynomial time approximation scheme} (APTAS) for the version of the problem with a constant number of scenarios; (iii) an exact method for the problem based on a branch-and-price algorithm for an exponential \emph{Integer Linear Programming} (ILP) model; and (iv) a \emph{Variable Neighborhood Search} (VNS) heuristic.

This paper is organized as follows. In Section~\ref{sec-related} we provide formal definitions to be used to contextualize the contributions of the paper.
In Section~\ref{sec-approx}, we show that an approximate solution of the BPPS can be obtained by solving an instance of the VBPP, leading to an absolute approximation algorithm for the BPPS\@.
In Section~\ref{sec-aptas}, we present an APTAS for the~BPPS, when the number of scenarios is constant. As for practical results, Section~\ref{sec-vns} describes a VNS algorithm for the BPPS and Section~\ref{sec-branchprice} presents a branch-and-price algorithm to solve an exponential model for the BPPS\@.
The evaluation of the proposed VNS and branch-and-price algorithms is performed in Section~\ref{sec-experiments}, from experiments based on the solution of randomly generated instances.
Finally, Section~\ref{sec-conclusions} presents the conclusions and directions for future research.

\section{Formal Definitions\label{sec-related}}
In this section, we provide some definitions to contextualize the contributions of this work. For simplicity, throughout the paper, we denote, for any natural $n$, the set $\{1, \dots, n\}$ simply by $[n]$.

%

\paragraph{Bin Packing Problem with Scenarios (BPPS)}

In the BPPS, we are given a number $d$ of scenarios, a set $\cali = [n]$ of $n$ items, with each $i \in \cali$ having a size $s_i \in \mathbb{Q}_+$ and a set of scenarios $\calk_i \subseteq [d]$, and an unlimited number of identical bins of capacity $W \in \mathbb{Q}_+$.
For each $k \in [d]$, the set of items in scenario~$k$ is denoted by $S_k = \{i \in \cali \mid k \in \calk_i \}$.
A solution $\calb$ of the BPPS is a partition of $\cali$ such that for each part $B \in \calb$ and for each scenario $k \in [d]$, $\sum_{i \in B \cap S_k} s_i \leq W$. The objective of the BPPS is to find a solution that minimizes
\[
\vbps(\calb) = \max_{k \in [d]}|\{B \in \calb: B \cap S_k \neq \emptyset\}|.
\]

The objective corresponds to minimize the number of bins of the worst-case scenario, i.e., the scenario with the largest number of bins. A closely related problem is the VBPP\@.

\paragraph{Vector Bin Packing Problem (VBPP)}

In the VBPP, we are given a number~$d$ of resources, a set $\cali = [n]$ of $n$ items, each $i\in \cali$ having a vector ${s_i \in \mathbb{Q}_+^d}$ of resource consumption, and an unlimited number of identical bins of resource capacity given as a $d$-dimensional vector ${W = (w_1, w_2, \dots, w_d)} \in \mathbb{Q}_+^d$.
A solution~$\calb$ is a partition of $\cali$ such that for each part $B \in \calb$
and each resource~$k \in [d]$, $\sum_{i \in B} s_{ik} \le w_k$.
The objective is to find a solution that minimizes
\[
\vvp(\calb) = |\calb|.
\]

The VBPP is a generalization of the BPP, and thus it is also \APX-hard. 
\citet{woeginger97} showed that there is no APTAS for the two-dimensional version of this problem unless $\PP=\NP$. 
\citet{chekuriK04} gave an asymptotic $O(\ln d)$-approximation for the case where $d$ is a fixed constant. Later on, this result was improved to $1 + \ln d$ by \citet{Bansal2009}. Also, as noted by \citet{Christensen2017}, the APTAS of \citet{VegaL81} implies a $(d + \eps)$-approximation for the VBPP\@.

%
%

For the theoretical results of this paper in Sections~\ref{sec-approx} and~\ref{sec-aptas}, we assume, without loss of generality, that the instances of the BPPS and the VBPP are normalized by a proper scaling of the items, so that the capacity of the bins correspond $1$ on each dimension.

An important concept in both theoretical and practical results for the BPP and related variants is the one of \emph{cutting pattern}: a feasible combination of items in a single bin.
For the BPPS, we represent a cutting pattern as a vector $p=(a_{1p},\dots,a_{np}, b_{1p},\dots,b_{dp})$ of binary coefficients, such that: each coefficient~$a_{ip}$ is equal to $1$ if and only if item $i$ is in pattern $p$, and each coefficient $b_{kp}$ is equal to $1$ if and only if some item of the scenario $k$ is in pattern~$p$. We denote by $\patternset$ the set of all feasible patterns for the BPPS\@.

\section{An Absolute Approximation Algorithm\label{sec-approx}}

In the following, we consider a mapping between the set of instances of the BPPS and a subset of instances of the VBPP\@. Starting with a normalized instance ${I = (d,\cali, \calk, s)}$ of the BPPS, we construct a normalized instance ${I'= (d, \cali, s')}$ of the VBPP\@. For that, let, for each $i \in \cali$ and $k \in [d]$,
\[
s_{ik} =
\begin{cases}
s_i & \mbox{if } i \in S_k,\\
0   & \mbox{otherwise.}
\end{cases}
\]
It is simple to verify that a solution~$\calb$ for~$I$ is a solution for~$I'$, and vice-versa. 

\begin{definition}
A solution $\calb$ is minimal for the BPPS if, for any two bins ${A, B \in \calb}$, there exists a scenario $k$, such that
\[
\sum_{i \in A \cap S_k} s_i + \sum_{i \in B \cap S_k} s_i > 1.
\]
\end{definition}

Notice that the value of $\calb$ for the BPPS is not larger than the value of $\calb$ for the VBPP\@. In the opposite direction, we have Theorem~\ref{lem1}.
\begin{theorem}\label{lem1}
If $\calb$ is a minimal solution for the BPPS, then $\vvp(\calb) \le \sqrt{d}\vbps(\calb)$. 
\end{theorem}

\begin{proof}

We say that two bins $A$ and $B$ are incompatible in a scenario $k$ if both bins, $A$ and $B$, are non-empty in scenario $k$. Let $G$ be a multigraph where each bin represents a vertex, and for each pair of bins~$A$~and~$B$ and scenario~$k$, there is one edge between~$A$ and~$B$ (labeled as $k$) if they are incompatible in~$k$.

For each scenario $k \in [d]$, we define $h_k = |\{ B \in \calb : B \cap S_k \ne \emptyset \}|$, i.e., $h_k$ is the number of bins used by $\calb$ in $k$. Let $h = \max_{k \in [d]} h_k$, and notice that, since, for each scenario, there is an edge for each two used bins, the number of edges in $G$ is
\[
|E(G)| = \sum_{k=1}^d \frac{h_k(h_k-1)}{2} \le d \frac{h(h-1)}{2}.
\]

Let $r$ be the number of bins. Since $\calb$ is minimal, any two bins are incompatible in at least one scenario and $G$ must have at least one edge between any two vertices, i.e., $|E(G)| \ge \frac{r(r-1)}{2}$. Thus,
\[
\frac{r(r-1)}{2} \le d \frac{h(h-1)}{2},
\] and, since $r \le dh$, as each bin is non-empty in at least one scenario, we have that
\[
r^2 \le d h^2 + r - dh  \le d h^2.
\]
The result follows as $r = \vvp(\calb)$ and $h = \vbps(\calb)$.\qedhere
\end{proof}

\begin{corollary}
Suppose there exists an $\alpha$-approximation for the VBPP, then there exists an $\alpha\sqrt{d}$-approximation for the BPPS\@.
\end{corollary}
\begin{proof}
    Consider an instance $I$ of BPPS and the instance $I'$ of VBPP obtained as described above. Also, let $\calb$ be the solution found by the $\alpha$-approximation for instance $I'$, $\calb_V^*$ be an optimal solution for VBPP with instance $I'$ and $\calb_S^*$ be a minimal optimal solution for BPPS with instance $I$. Then, we have \[\vbps(\calb) \leq \vvp(\calb) \leq \alpha \vvp(\calb^*_V) \leq \alpha  \vvp(\calb^*_S) \leq \alpha \sqrt{d} \vbps(\calb^*_S).\qedhere\]
\end{proof}

The bound given by this strategy cannot improve the dependency on $d$, by the following lemma.

\begin{theorem}
For each $d \ge 1$, there exists an instance $I$ of the BPPS with $d$ scenarios and a solution~$\calb$ for $I$, which is minimal for the BPPS, such that ${\vvp(\calb) \geq \frac{\sqrt{d}}{2} \vbps(\calb)}$.
\end{theorem}

\begin{proof}
Let $r = \left\lceil \sqrt{d} \right\rceil$.
We construct an instance $I = (d, \cali, \calk, s)$ that contains~$r$ items.

We consider a complete graph $G$ with $r$ vertices. For each vertex $i$ of $G$, create an item~$i$ with size $s_i = 1$, and for each edge $\{i, j\}$ of $G$, create a scenario~$k$ with items $S_k = \{i, j\}$. Notice that each item is in exactly $r - 1$ scenarios, and that the number of scenarios is
\[
|E(G)| = \frac{r(r-1)}{2} = \frac{\lceil \sqrt{d} \rceil(\lceil \sqrt{d}
\rceil-1)}{2} \le \frac{(\sqrt{d} + 1) \sqrt{d}}{2} \le d.
\]

Now, we create a solution $\calb$ with $r$ bins, such that, for each item $i$, there is one bin $B_i = \{i\}$.
We claim that $\calb$ is minimal for the BPPS\@. Indeed, for any two bins $B_i$ and $B_j$, items $i$ and $j$ are contained in a common scenario $k$, with $S_k = \{i, j\}$, and $s_i + s_j > 1$.

Notice that $\vvp(\calb) = r$. Let $h = \vbps(\calb)$, so $h$ is the maximum number of non-empty bins in a given scenario. Since each scenario contains exactly $2$ items, then $h = 2$. Therefore,
\begin{align*}
r = \lceil \sqrt{d} \rceil \ge \frac{\sqrt{d}}{2} 2 = \frac{\sqrt{d}}{2} h.\tag*{\qedhere}
\end{align*}
\end{proof}

Finally, using the $1 + \ln d$-asymptotic approximation algorithm of \citet{Bansal2009} for the VBPP (for constant $d$) and the fact that the APTAS of \citet{VegaL81} implies a $(d + \eps)$-approximation for the VBPP~\cite{Christensen2017}, we have the following result.

\begin{corollary}
    There exists a $(d + \eps)\sqrt{d}$-approximation for the BPPS\@. Also, for any constant~$d$, there exists a $(1 + \ln d)\sqrt{d}$-approximation for the BPPS\@.
\end{corollary}

\section{An Asymptotic Polynomial Time Approximation Scheme\label{sec-aptas}}
We present an APTAS for the BPPS when the number of scenarios $d$ is a constant.
First, we observe that, if all items are large and the number of distinct item sizes is bounded by a constant, then the number of valid patterns of items in a bin is also bounded by a constant. It implies that for this restricted case, a polynomial algorithm can be readily obtained by simply enumerating the frequencies of all patterns.

To obtain such a restricted instance, we use the \emph{linear grouping} technique, following~\citet{VegaL81}, by grouping items of similar sizes and creating a map from smaller to larger items. A naive approach does not work, however, since one must take into account the scenarios. Thus, we group items according to their scenarios separately, but we do the mapping simultaneously.

Finally, to solve a general instance, we combine small items into artificial large items. Again, this is done on a per-scenario basis, so as only compatible items are combined. If the small items do not fit exactly into the area of the artificial items, then this change may cause bins to be slightly augmented. The surpassing items are then relocated to new bins greedily, increasing the number of bins by just a fraction of the optimal value.

In the following, we use the notion of type. The \emph{type} of item $i$ is the set of scenarios which contain $i$. The set of all types is denoted by $\calt$.
Also, the value of an optimal solution for an instance $I$ is denoted by $\opt(I)$.

\subsection{Restricted instances}

In this subsection, we consider instances of the BPPS in which the items are large and the number of distinct sizes is bounded by a constant. Precisely, let~$V$ be a set of positive numbers and consider a constant $\eps$, with ${0 < \eps \le 1/4}$ and~$1/\eps$ is an integer. An instance $(d, \cali, \calk, s)$ of the BPPS  is said to be \mbox{\emph{$V$-restricted}} if $\{s_i : i \in \cali\} = V$ and $s_i \ge \eps^2$ for every~$i \in \cali$.

Recall that a bin $B \subseteq \cali$ is feasible if, for any scenario $k \in [d]$, ${\sum_{i\in B \cap S_k}s_i\leq 1}$.
For a type $t \in \calt$ and $v \in V$, let $B_{tv}$ be the subset of items in $B$ with type $t$ and size $v$. Observe that the family of all sets $B_{tv}$ form a partition of $B$. The \emph{pattern} of $B$ is the vector ${(|B_{tv}|)}_{t \in \calt, v \in V}$.

\begin{lemma}\label{lemma:amount_patterns}
If $|V|$ is constant, the number of distinct patterns for a $V$-restricted instance is bounded by a constant.
\end{lemma}

\begin{proof}
Consider a feasible bin $B$. Since $s_i \ge \eps^2$, for any $i \in B$,
\[
|B| \eps^2 \le \sum_{i \in B} s_i \le \sum_{k=1}^d \sum_{i \in B \cap S_k} s_i \le \sum_{k=1}^d 1 = d,
\]
where the last inequality follows because $B$ is feasible. It means that the number of items in any bin is at most $d/\eps^2$. Since a vector corresponding to a pattern has $|\calt||V|$ elements, the number of patterns is at most ${(1+d/\eps^2)}^{|\calt||V|}$, which is constant since $d$ is constant.\qedhere
\end{proof}

\begin{lemma}\label{lemma:constant_size_opt}
If $|V|$ is constant, then an optimal solution of a $V$-restricted instance can be computed in polynomial time.
\end{lemma}

\begin{proof}
Consider a $V$-restricted instance $I$, and let $M$ be the set of distinct patterns for $I$. Given a solution $\calb$ and a pattern $p \in M$, the number of bins in~$\calb$, which has (a non-empty) pattern $p$, is denoted by~$n_p$.
Clearly $n_p \le |\cali|$ for every $p \in M$, as there are only $|\cali|$ items, and each bin has at least one item.
The configuration of $\calb$ is the vector ${(n_p)}_{p \in M}$.
Thus, the number of possible configurations is bounded by $|\cali|^{|M|}$, which is polynomial since, by Lemma~\ref{lemma:amount_patterns},~$|M|$ is constant.

Notice that, given a vector ${(n_p)}_{p \in M}$, one can either find a solution $\calb$ with this configuration, or decide there is no such solution. Indeed, it is sufficient to count the total number of items of each type and size over all the patterns.
Therefore, one can list all configurations in polynomial time and return the one with the minimum value.\qedhere
\end{proof}

\subsection{An APTAS for large items}

Suppose now that we have an instance $I = (d, \cali, \calk, s)$, such that all items are large, i.e., for each $i \in \cali$, we have $s_i \ge \eps^2$, but the number of distinct sizes is not necessarily bounded by a constant.

\begin{definition}
Given instances $I = (d, \cali, \calk, s)$ and $\bar I = (d, \bar \cali, \bar \calk, \bar s)$, we say that~$I$ dominates~$\bar I$ if there exists a subset $\cali' \subseteq \cali$ and a bijection $f : \cali' \to \bar \cali$, such that, for every $i \in \cali'$, items~$i$ and~$f(i)$ have the same type and $s_i \ge \bar s_{f(i)}$. In this case, we write $I \succcurlyeq \bar I$.
\end{definition}

\begin{lemma}\label{lemma:dominates}
If $I \succcurlyeq \bar I$, then $\opt(I) \ge \opt(\bar I)$.\qed
\end{lemma}

In the following, we create an instance $\bar{I} = (d, \bar\cali, \bar \calk, \bar s)$ with $I \succcurlyeq \bar I$, such that all items of~$\bar I$ are large and the number of distinct sizes is bounded by a constant.

First, for each $t \in \calt$, let $\cali_t$ be the set of items of type $t$ and let $m = \frac{2^d}{\eps^3} - 1$. We consider the set $\cali_t$ sorted in decreasing order of size, and build $m + 1$ groups of consecutive items, obtaining a partition $\{G^0_t, G^1_t,\dots,G^{m}_t\}$, where the first $m$ groups have size $\lceil |\cali_t|/(m + 1) \rceil$, and the last group either has the same size, or is smaller.

Now, for each $t \in \calt$ and $\ell \in \{1,2,\dots,m\}$, we create a set $\bar G^\ell_t$ as follows: for each item $i \in G^\ell_t$, add an item to $\bar G^\ell_t$ with the same set of scenarios $t$, and with size equal to the size of the largest item in $G^\ell_t$.
Let $\bar \cali$ be the union of all sets $\bar G^\ell_t$, and $\bar I$ be the instance induced by~$\bar\cali$.

\begin{lemma}
$I \succcurlyeq \bar I$.
\end{lemma}

\begin{proof}
For $\ell \in \{1, 2, \dots, m\}$, as $|G_t^{\ell-1}| \ge |G_t^\ell| = |\bar G_t^{\ell}|$, we can find a bijection~$f$ between a subset of $G_t^{\ell-1}$ and $\bar G_t^{\ell}$ for $\ell \in \{1, 2, \dots, m\}$ such that $i$ and $f(i)$ has the same type. Finally, as $s_i \geq s_j$ for $i \in G_t^{\ell-1}$ and $j \in G_t^{\ell}$ and $\bar s_{f(i)}$ is the maximum size of an item in  $G_t^{\ell}$, we have that $s_i \geq \bar s_{f(i)}$.\qedhere
\end{proof}

For a type $t \in \calt$, the group of the largest items is $G^0_t$. For each item $i$ in this group, we pack it into a separate bin $\{i\}$. By joining all these bins, we obtain a packing $\calp^0$.

\begin{lemma}\label{lemma:g0}
$\vbps(\calp^0) \le \eps \opt(I) + 2^d$.
\end{lemma}

\begin{proof}
Let $t \in \calt$ and fix an arbitrary scenario $k$ of $t$.
We consider an optimal solution for~$I$, and let $\calp_k$ be the set of bins used in scenario $k$ in this solution. Therefore, $|\calp_k| \le \opt(I)$. Observe that, since $k$ belongs to $t$, the set of items included in bins of $\calp_k$ must contain $\cali_t$. As, for each $i \in \cali_t$, $s_i \ge \eps^2$,
\[
\eps^2 |\cali_t| \le \sum_{i \in \cali_t} s_i \le
\sum_{B \in \calp_k} \sum_{i \in B \cap S_k} s_i \le
\sum_{B \in \calp_k} 1 = |\calp_k| \le \opt(I).
\]

This implies that the value of $\calp^0$ can be bounded by
\[
\sum_{t \in \calt} |G^0_t| =
\sum_{t \in \calt} \left\lceil \frac{\eps^3}{2^d}\, |\cali_t| \right\rceil \le
2^d + \sum_{t \in \calt} \frac{\eps \opt(I)}{2^d} \leq \eps \opt(I) + 2^d.\tag*{\qedhere}
\]
\end{proof}

To obtain a packing of the remaining items, we solve instance~$\bar I$. For this, let~$V$ be the set of distinct sizes in $\bar I$. For each $t \in \calt$, we created $m$ groups, where each group has items of the same size. Therefore, ${|V| \le 2^d m}$.
This means that $\bar I$ is a $V$-restricted instance for $|V|$ bounded by a constant. Using Lemma~\ref{lemma:constant_size_opt}, we obtain a solution~$\bar\calp^1$ for $\bar I$ in polynomial time.
Since $I \succcurlyeq \bar I$, we can obtain a solution $\calp^0 \cup \calp^1$ for $I$, where~$\calp^1$ is obtained from $\bar \calp^1$ by replacing items of $\bar G^\ell_t$ by items of $G^\ell_t$. Let $I^1$ be the instance obtained from the items of $\calp^1$. Observe that $\calp^1$ is feasible, since we replaced items of larger size which appear in the same set of scenarios, thus for each scenario, the capacity of each bin remains respected.

\begin{lemma}\label{lemma:g1t}
$\vbps(\calp^1) \le \opt(I)$.
\end{lemma}

\begin{proof}
Note that the number of unused bins in each scenario is the same for $\calp^1$ and $\bar \calp^1$, thus $\vbps(\calp^1) = \vbps(\bar\calp^1)$. Since, $\bar\calp^1$ is optimal for $\bar I$, we have $\vbps(\bar\calp^1) = \opt(I^1)$. Now, using Lemma~\ref{lemma:dominates} and the fact that $I \succcurlyeq \bar I$, we have $\opt(I^1) \le \opt(I)$.\qedhere
\end{proof}

Combining the previous results, one obtains a APTAS for instances with large items only.

\begin{lemma}\label{lemma:large_items}
Suppose for every $i \in \cali$, $s_i \ge \eps^2$. Then one can find in polynomial time a packing $\calp$ of $\cali$ with $\vbps(\calp) \le (1 + \eps) \opt(I) + 2^d$.
\end{lemma}

\begin{proof}
Define $\calp = \calp^0 \cup \calp^1$. Notice that $\calp$ contains every item of $\cali$, and thus it is feasible. Using lemmas~\ref{lemma:g0}~and~\ref{lemma:g1t}, \[\vbps(\calp) \le \vbps(\calp^0) + \vbps(\calp^1) \le  \eps \opt(I) + 2^d + \opt(I).\qedhere\]
\end{proof}

\subsection{An APTAS for the general case}

In the general case of the BPPS, an instance $I = (d, \cali, \calk, s)$ may contain large and small items.
Let $\call = \{i \in \cali : s_i \ge \eps^2 \}$ be the set of large items, and $\cals = \cali \setminus \call$ be the set of small items.

To obtain a packing of $\call \cup \cals$, we first replace $\cals$ by the set of large items,~$\hat \cals$, which we define as follows.  For each type $t \in \calt$, let $\cals_t$ be the set of small items with type $t$. Now, let $\hat\cals_t$ be a set of $\lceil \sum_{i \in \cals_t} s_i / (\eps - \eps^2) \rceil$ items, such that each $j \in \hat\cals_t$ has size $\hat s_j = \eps$ and is of type~$t$. The set $\hat \cals$ is the union of sets $\hat \cals_t$ for every type $t$.

Thus, we create an instance $\hat I = (d, \hat \cali, \hat \calk, \hat s)$ whose set of items is $\hat \cali = \call \cup \hat\cals$, each of which has size at least $\eps^2$. Lemma~\ref{lemma:modopt} compares the optimal value of $I$ and $\hat I$.

\begin{lemma}\label{lemma:modopt}
$\opt(\hat I) \le (1 + 2^{d+1} (d + 1)\eps)\opt(I) + 1$.
\end{lemma}

\begin{proof}
Consider an optimal solution $\calp^*$ for $I$. For each bin $B \in \calp^*$ and type $t \in \calt$, we define $B_t = B \cap \cals_t$, i.e., $B_t$ is the set of small items with type $t$. Then, we replace the items in $B_t$ by $\lceil \sum_{i \in B_t} s_i / (\eps - \eps^2)\rceil$ new items of size $\eps$, and with the same scenarios of~$t$. We call the modified packing by $\calp'$.

Observe that $\calp'$ may be infeasible, since new items may surpass the bins' capacity. We can obtain a feasible solution $\calp''$ by relocating some created items from $\calp'$ to new bins. For each type~$t$ and each bin~$B \in \calp'$, the number of items needed to be picked is the additional area divided by $\eps$ rounded up, that is, 
\begin{align*}
\left\lceil
  \frac{
    \left\lceil 
      \sum_{i \in B_t} \frac{s_i}{\eps - \eps^2}
    \right\rceil \eps 
    - \sum_{i \in B_t} s_i
  }{
    \eps
  }
\right\rceil
& \leq
1 + 
\frac{
  \eps + 
  \eps \sum_{i \in B_t} \frac{s_i}{\eps - \eps^2}
  - \sum_{i \in B_t} s_i
}{
  \eps
}\\
& =
2 + 
\frac{1}{1 - \eps}
\sum_{i \in B_t} s_i \\
& \leq
2 + 2\sum_{i \in B_t} s_i
\end{align*}
where the last inequality follows from the fact that $\eps \leq 1/2$. Let $R$ be the set of all picked items. Observe that ${|\calp^*| \le d \opt(I)}$, since each bin is not empty for at least one scenario. Thus, as there are $2^d$ types and ${\sum_{B \in \calp^*} \sum_{i \in B_t} s_i \leq \opt(I)}$, the number of items in $R$ is at most
\begin{align*}
\sum_{t \in \calt} \sum_{B \in \calp^*}  \left(2 + 2\sum_{i \in B_t} s_i\right) 
& = 2^{d+1} |\calp^*| + 2^{d+1} \opt(I)\\
& = 2^{d+1} d \opt(I) + 2^{d+1} \opt(I)\\
& = 2^{d+1} (d + 1) \opt(I).
\end{align*}

Since $1/\eps$ is an integer, one can rearrange all items of $R$ into at most \[\left\lceil \frac{2^{d+1} (d + 1) \opt(I)}{1/{\eps}}\right\rceil\] new bins of unit capacity. Let $\calr$ be this set of new bins, and obtain a feasible packing $\calp''$ of $I$, by making a copy of $\calp'$, removing the items of $R$, and adding the bins of $\calr$.

Notice that $\vbps(\calp') = \vbps(\calp^*) = \opt(I)$. Using these facts, the value of $\calp''$ can be estimated as
\begin{align*}
\vbps(\calp'')
  &\le \vbps(\calp')  + |\calr|\\
  &= \opt(I) + \left\lceil \frac{2^{d+1} (d + 1) \opt(I)}{1/{\eps}}\right\rceil\\
  &= \opt(I) + 2^{d+1} (d + 1)\eps \opt(I) + 1.
\end{align*}

We notice that $\calp''$ contains all large items $\call$ and, for each type $t$, a certain number of items with scenarios~$t$ and size~$\eps$. Thus, to obtain a packing for $\hat \cali$ from $\calp''$, it is sufficient to show that for each $t$, the number of created items corresponding to $t$ is not smaller than~$|\hat \cals_t|$. Indeed, for a given $t$, the number of created items is 
\[
\sum_{B \in \calp^*}  \left\lceil \sum_{i \in B_t} s_i / \eps \right\rceil \ge
\left\lceil  \sum_{i \in \cals_t}  s_i / \eps \right\rceil = |\hat \cals_t|.
\]

To conclude the proof, we create a packing~$\hat \calp$ for~$\hat \cali$ by removing the exceeding items of~$\calp''$. Since removing such items can only decrease the objective value, ${\vbps(\hat \calp) \le \vbps(\calp'')}$.\qedhere
\end{proof}

Finally, to obtain a packing for $I$, one can first obtain a packing of~$\hat I$, then replace items from $\hat \cals_t$ by items in $\cals_t$. This leads to Theorem~\ref{teo:aptas}.

\begin{theorem}\label{teo:aptas}
For every constant $\eps' > 0$, one can find in polynomial time a packing $\calp$ of $\cali$ such that $\vbps(\calp) \le (1 + \eps') \opt(I) + 2 + 2^d$. 
\end{theorem}

\begin{proof} 
Define $r = 2^{d+1} (d + 1)$ and define $\eps$ as the largest number such that $\eps \leq \min(\frac{\eps'}{3r}, 1/4)$ and $1/\eps$ is an integer.
Let $\hat I$ be the instance obtained from $I$ by replacing, for each type, the items smaller than $\eps^2$ with items of size $\eps$, as previously mentioned. Since each item of $\hat I$ has size at least $\eps^2$, by Lemma~\ref{lemma:large_items}, we obtain a packing $\hat \calp$ for $\hat I$ with value $\vbps(\hat \calp) \le {(1 + \eps) \opt(\hat I)} + 2^d$ in polynomial time.

For each $t \in \calt$, we place all items of $\cals_t$ over the items of $\hat \cals_t$. Precisely, we do the following: start by making a copy $\calp$ of $\hat\calp$; select a set $A$ of unpacked items in~$\cals_t$ whose total size is at least $\eps - \eps^2$ and at most $\eps$. If there is not such a set $A$, then let $A$ be the remaining set of unpacked items in~$\cals_t$. Find a bin~${B \in \calp}$ with an item $j \in \hat \cals_t\cap B$ and replace $j$ by the items of $A$. Since there are $\lceil \sum_{i \in \cals_t} s_i / (\eps - \eps^2) \rceil$ items in $\hat \cals_t$, and each group~$A$, except perhaps one of them, has size at least $\eps - \eps^2$, we always find an unused item $j \in \hat \cals_t$; repeat these steps while there are unpacked items in $\cals_t$ for some $t$.

The resulting packing $\calp$ contains all items in $\cali$. As no bin was added, $\calp$ has the same value of $\hat \calp$. Using Lemma~\ref{lemma:modopt}, we finally obtain
\begin{align*}
\vbps(P) &= \vbps(\hat P) \le (1 + \eps) \opt(\hat I) + 2^d \\
         &\le (1 + \eps) [ (1 + 2^{d+1} (d + 1)\eps)\opt(I) + 1 ] + 2^d \\
         &= (1 + r\eps^2 +r\eps + \eps)\opt(I) + 2^d + 1 + \eps\\
         &\leq (1 + 3r\eps)\opt(I) + 2^d + 1 + \eps\\
         &\le (1 + \eps') \opt(I) + 2 + 2^d.\tag*{\qedhere}
\end{align*}
\end{proof}

\section{A Variable Neighborhood Search Algorithm\label{sec-vns}}
The Variable Neighborhood Search is a metaheuristic that searches in different neighborhood structures and has systematic change of neighborhood to find better solutions and escape from local optima. \citet{mladenovicH97} proposed the VNS not only for combinatorial optimization problems, but for optimization in general. Due to its high applicability to similar problems \citep{puchingerR08,fleszarH02,beltranCCM04}, we apply the VNS to solve the BPPS\@.

For successfully implement it, we need to clarify three key characteristics of the problem: the local search neighborhood structures and how to navigate the search space; sensitive and computationally viable objective functions; and the stopping criterion. In the following, we detail each of these and describe the resulting algorithm.

\subsection{Neighborhood structures}\label{subsec:structures}

A neighborhood structure specifies a well-defined way to get from any given solution into another solution that is ``close'' to the first one. We refer to these neighborhood structures as a  ``movement'' that takes one solution as input and, depending on some structure-dependent input parameters, it returns a different but ``close'' solution.
We define four neighborhood structures for the BPPS in the following order:

\begin{itemize}
    \item $N_1$: Given items $i$ and $j$ packed into different bins, move $i$ to the bin currently containing $j$, then move $j$ to the bin that contained $i$; 
    \item $N_2$: Given an item $i$ and a bin $B$ that currently does not contain $i$, move $i$ to $B$; 
    \item $N_3$: Given bins $B_1, B_2$, and $B_3$ and items $i$ and $j$ packed into $B_1$, move these items to $B_2$ and $B_3$, respectively; 
    \item $N_4$: Given a bin $B$, remove it from the solution and repack the items in the first bin that can accommodate each of them, respecting the order that they appeared in $B$. 
\end{itemize}

Neighborhoods $N_1, N_2$, and $N_3$ are prioritized according to their complexity. Although the neighborhood structure $N_4$ is the smallest, it can take comparatively more computational effort to compute since we have to repack all items of a given bin, so we consider it as the last one.

\subsection{Objective function}

The objective function for the BPPS may not be sensitive enough to rank two neighbor solutions. Thus, to guide the VNS, we define a fitness function that penalizes solutions in which the bin occupation per scenario is small.

Let $\calb$ be a solution for the BPPS and $\calb_k$ be the subset of bins that contain items in scenario~$k$, where $|\calb_k|$ is the number of bins used in $k$. Furthermore, let $\text{ocp}(k,B)$ be the total size of the items in bin $B$ from scenario $k$. Given a solution $\calb$, the VNS moves to a neighbor solution $\calb'$ such that $f(\calb) > f(\calb')$, for $f$ defined as:
\begin{align*}
f(\calb) &=  \vbps(\calb) - \sum_{k \in \calk} \sum_{B \in \calb_k} {\left(\frac{\text{ocp}(k,B)}{|\calb_k| W}\right)}^2
\end{align*}

\subsection{Stopping criterion}

Since the VNS cannot tell whether an optimal solution is reached, we use two stopping criteria to decide when the algorithm stops. The first criterion is the \emph{timeout}, where a preset CPU time limit ($t_{\max}$) is imposed, so the VNS stops as soon as the time limit is exceeded. The second criterion is the \emph{local-convergence}, where a preset number of iterations ($c_{\max}$) is imposed, so the VNS stops if the best known solution is not updated within this number of iterations.

\subsection{Algorithm}

The VNS takes as input an initial solution $x$, a number of neighborhood structures $N_{\max}=4$, a time limit $t_{\max} = 1800$ seconds and a convergence limit $c_{\max} = 500$ iterations (these values were achieved after preliminary computational tests). The initial solution $x$ is obtained as follows: items are sorted in non-increasing order of their sizes and then packed in the first open bin where it fits, considering the capacity constraints in every scenario and opening a new bin whenever necessary. 

The overall structure of the proposed VNS is given in the following algorithm:
{\small
\begin{codebox}
  \Procname{$\proc{VNS}(x,N_{\max},t_{\max},c_{\max})$}
    \li $t \gets 0$,
    $c \gets 0$
    \li \While $t < t_{\max}$ and $c < c_{\max}$
    \li     \Do
                $\kappa \gets 1$
    \li         $improvement \gets \text{False}$                \>\>\>\>\>\>\> \Comment Assume there will be no improvement
    \li         \Repeat
    \li             $x' \gets \proc{Shake}(x,\kappa)$                \>\>\>\>\>\>   \Comment Get a random solution in $N_\kappa(x)$
    \li             $x'' \gets \proc{Local-Search}(x',\mathcal{N}_{\max})$ \>\>\>\>\>\>   \Comment Perform VND to improve $x'$
    \li             \If $f(x'') < f(x)$                         \>\>\>\>\>\>   \Comment There has been an improvement
    \li                 \Then
                            $x \gets x''$
    \li                     $\kappa \gets 1$
    \li                     $improvement \gets \text{True}$
    \li                 \Else                                   \>\>\>\>\>     \Comment Local optima, change neighborhood
    \li                     $\kappa \gets \kappa+1$
                        \End
    \li         \Until $\kappa \isequal N_{\max}$
    \li         \If $improvement$                               \>\>\>\>\>\>\> \Comment There has been an improvement
    \li            \Then $c \gets 0$
    \li         \Else                                           \>\>\>\>\>\> \Comment No improvement in this iteration
    \li             $c \gets c + 1$
                \End
    \li         $t \gets \proc{CPU-Time}()$                     \>\>\>\>\>\>\> \Comment Update processing time.
            \End
\end{codebox}
}

The VNS uses a \proc{Shake} procedure that simply returns a solution $x'$ randomly selected from the neighborhood $N_\kappa$ of $x$. The \proc{Local-Search} procedure is based on the Variable Neighborhood Descent (VND) approach, which is the deterministic descent-only version of the VNS\@. Our local search  takes as input a solution $x$ and the number of neighborhood structures $N_{\max}=4$, resulting in the following algorithm:
{\small
\begin{codebox}
    \Procname{$\proc{Local-Search}(x,N_{\max})$}
    \li $\kappa \gets 1$
    \li \Repeat
    \li     $x' \gets \arg\min_{y \in N_\kappa(x)} f(y)$ \>\>\>\>\>\> \Comment Find the first best neighbor of $x$
    \li     \If $f(x') < f(x)$                                  \>\>\>\>\>\> \Comment There has been an improvement
    \li         \Then
                    $x \gets x'$
    \li             $\kappa \gets 1$
    \li         \Else                                           \>\>\>\>\> \Comment Local optima, change neighborhood
    \li             $\kappa \gets \kappa+1$
                \End
    \li     \Until $\kappa \isequal N_{\max}$
\end{codebox}
}

\section{A Branch-and-Price Algorithm\label{sec-branchprice}}
We propose an exact method for the BPPS based on column generation and branch-and-bound, that is, a branch-and-price algorithm. This method consists of modelling the problem as an ILP with a very large (usually exponential) number of variables, which are generated and added to the model only as they are needed, in each node of the branch-and-bound tree. Branch-and-price has been used successfully to solve practical instances of the BPP and several of its variants \citep{DelormeI16,sadykovV13,dellamicoFI20,borgesMSX20,bettinelliCR10}. 

One of the most successful formulations for the classical BPP and variants is based on associating a variable to each cutting pattern (see, e.g., \citet{DelormeI16}).
In practice, such kinds of formulations lead to models with an exponential number of variables, as the number of cutting patterns is usually exponential with respect to the number of items.
Due to this, to solve practical instances with such models, one has to rely on column generation to solve the corresponding linear relaxation.

The \emph{column generation method}, proposed by \citet{FF58} and generalized by \citet{DW61}, solves the linear relaxation of models with a large number of variables. \citet{GilmoreG61, GilmoreG63} were the first to present practical experiments based on column generation, by solving the linear relaxation of an exponential model for the BPP\@. The column generation method solves the linear relaxation by iteratively solving the \emph{restricted master problem} (RMP), which considers only a subset of variables (columns) of the original model. At each iteration, the method has to determine the existence of non-basic columns which are candidate to improve the current solution of the RMP, i.e., columns with negative reduced cost (on minimization problems). If the method determines that no such column exists, then the solution of the RMP is an optimal solution for the original linear relaxation, providing a bound for the integer problem.

\subsection{A pattern-based model}
%
%
Recalling that $\patternset$ is the set of all cutting patterns for the BPPS, the following exponential ILP model solves the BPPS\@:
\begin{align}
\label{CGModel:ObjectiveFunction}
\mathrm{minimize}\ \ \ & \calf, &  \\
\label{CGModel:ItemsConstraints}
\mathrm{s.t.:}\ \ \ & \displaystyle \sum_{p \in \patternset} a_{ip} X_p \geq 1, & \forall i\in \mathcal{I};\\
\label{CGModel:ScenariosConstraints}
&  \displaystyle \sum_{p \in \patternset} b_{kp} X_p \leq \calf, & \forall k\in \mathcal{K};\\
\label{CGModel:11}
& X_p \in \{0,1\}, & \forall p\in \patternset;  \\
\label{CGModel:22}
& \calf \in \mathbb{Z}_+. & 
\end{align}

For each pattern $p\in \patternset$, binary variable $X_{p}$ is equal to $1$ if and only if $p$ is used in the solution. The integer variable $\calf$ represents an upper bound on the number of bins (patterns) that is used by each individual scenario. The objective function~\eqref{CGModel:ObjectiveFunction} is to minimize $\calf$. Constraints~\eqref{CGModel:ItemsConstraints} guarantees that every item belongs to at least one cutting pattern of the solution. Constraints~\eqref{CGModel:ScenariosConstraints} ensures that, for every scenario $k\in \mathcal{K}$, the number of bins that packs at least one item from $k$ is at most $\calf$. The domain of variables is in~\eqref{CGModel:11} and~\eqref{CGModel:22}.

Next, we describe a column generation algorithm to solve the linear relaxation of~\eqref{CGModel:ObjectiveFunction}--\eqref{CGModel:22}. As the initial set of columns for the RMP, we consider each cutting pattern with a single item. The dual solution of the RMP consists of a value $\alpha_i \geq 0$ for each item $i$ associated with a constraint in~\eqref{CGModel:ItemsConstraints}, and a value $\beta_k\leq 0$ for each scenario $k$ associated with a constraint in~\eqref{CGModel:ScenariosConstraints}. The dual solution is used to solve the {\em pricing problem}, which determines a column with minimum reduced cost, that is
\begin{equation}
\nonumber
\min_{p\in \patternset} \left\{ c_p-\left( \sum_{i\in \mathcal{I}}\alpha_i a_{ip} + \sum_{k\in K}\beta_k b_{kp}\right)\right\},
\end{equation}
which is equivalent to
\begin{equation}
\label{CGModel:ModifiedPricingProblem}
\max_{p\in \patternset} \left\{ \sum_{i\in \mathcal{I}}\alpha_i a_{ip}+\sum_{k\in K}\beta_k b_{kp}\right\}.
\end{equation}


\subsection{The pricing problem}\label{Sec:PricingProblem}
In order to generate a new column, we need to find a feasible pattern that optimizes~\eqref{CGModel:ModifiedPricingProblem}. Since it includes a term for each item and a term for each scenario used in the pattern, we can describe the pricing problem as the following Knapsack Problem with Scenarios.

\paragraph{Knapsack Problem with Scenarios (KPS)}
In the KPS, we are given: a number $d$ of scenarios; a set $\cali = [n]$ of $n$ items, with each item $i\in \cali$ having a size $s_i$, item value $v_i \in \mathbb{Q}$, and a set of scenarios $\calk_i \subseteq [d]$; a vector $u \in \mathbb{Q}^d$ of scenario values, and a knapsack of capacity $W \in \mathbb{Q}_+$. For simplicity, for each item $i\in \mathcal{I}$, define:
\begin{equation*}
  s_i^{k} = \begin{cases}
  s_i, \text{ if item } i \text{ is in the scenario }k;\\
  0, \text{ otherwise}.
  \end{cases}
\end{equation*}
A solution is a subset $B$ of $\cali$ such that for each scenario $k \in \calk$, $\sum_{i \in B} s_i^k \le w_k$.
The objective is to find a solution that maximizes
\[
\sum_{i \in B} v_i + \sum_{k: s_i^k > 0,\, i \in B} u_k.
\]

\begin{theorem}\label{theo:kps}
  KPS is $\NP$-hard in the strong sense.
\end{theorem}

\begin{proof}
  Let us recall the Maximum Independent Set problem, which given a graph $G$ consists of finding a subset of maximum cardinality of its vertices such that no two of its elements are adjacent in $G$. This problem is known to be strongly \NP-hard \citep{gareyJ78}, so all we need to do is present a polynomial reduction from this problem into the KPS\@.

  Let $I = G(V,E)$ be an instance of the Maximum Independent Set problem. We construct an instance $I' = (d, \mathcal{I}, s, v, \calk, u, W)$ of the KPS such that $|\mathcal{I}| = |V|$, $d = |V|$. For each item $i \in \mathcal{I}$ there is a scenario $k_i \in [d]$ associated with it. We define
  $\calk_i = \{k_i\} \cup \{k_j \colon (i, j) \in E\}$ and we let $s_i^k = 1$ if $k \in \calk_i$. We define $v_i = 1$ for all $i \in \mathcal{I}$, $u_k = 0$ for all $k \in [d]$ and $W = 1$. 
  
  Now notice that there is a bijection between solutions of $I$ and $I'$ which preserves the value of the objective function two items corresponding to two adjacent vertices of $V$ cannot be in the same solution of $I'$, the items are unit-valued and the scenarios are zero-valued.\qed
\end{proof}

As a consequence of Theorem~\ref{theo:kps}, our pricing subproblem does not admit a pseudo-polynomial time algorithm, unless $P=\NP$, thus, we propose the following ILP model to solve it. 

We can use the dual solution $(\alpha, \beta)$ of the RMP as item-values and scenario-values $(u,v)$ to describe the following model for the pricing subproblem
\begin{align}
\label{PricingModel:1}
\mathrm{maximize}\ \ \ & \sum_{i\in \mathcal{I}}\alpha_i A_i+\sum_{k\in \mathcal{K}}\beta_k B_k &   \\
\label{PricingModel:Constraints1}
\mathrm{s.t.:}\ \ \ & \displaystyle \sum_{i \in \mathcal{I}} s_i^{k} A_i \leq W B_k, & \forall k\in \mathcal{K};\\
\label{PricingModel:3}
& A_i, B_k \in \{0,1\}, & \forall i\in \mathcal{I}, k\in \mathcal{K} .
\end{align}

In model~\eqref{PricingModel:1}-\eqref{PricingModel:3}, variables $A_i$ are related to the coefficients $a_{ip}$, and variables $B_k$ are related to the coefficients $b_{kp}$. Hence, variables $A$ and $B$ determine the cutting pattern that is generated. The objective function~\eqref{PricingModel:1} is related to the pricing formula in~\eqref{CGModel:ModifiedPricingProblem}.  Constraints~\eqref{PricingModel:Constraints1} assures that the solution of the model represents a valid cutting pattern for the BPPS\@. Notice that $W$ is the capacity of a bin, then if a scenario $k\in \mathcal{K}$ is used ($B_k=1$), the sum of the sizes of all items that belong to the scenario $k$ must be at most $W$, and if this scenario is not used ($B_k=0$), then the sum of the sizes of all items that belong to $k$ must be equal to~$0$.

\subsection{The branch-and-price algorithm}

To solve the model~\eqref{CGModel:ObjectiveFunction}-\eqref{CGModel:22}, we propose a branch-and-price algorithm. The branch-and-price is based on an enumeration of the fractional solution of the column generation algorithm.
For this end, we use the branching scheme of \citet{ryanF81}. Given a fractional solution $X^*$, we know, from \citet{vanceBHN94}, that there exists rows $l$ and $m$ such that 
\[
  0 < \sum_{p:a_{lp}=1, a_{mp}=1} X_p < 1.
\]
We use this pair of rows to create two branches for the current node. On one side we enforce that items $l$ and $m$ must be packed in the same pattern, and on the other side we enforce that they must be packed in different patterns. This is achieved by the following branching constraints
\begin{align}
\sum_{p:a_{lp}=1, a_{mp}=1} X_p \geq 1  \label{CGModel:GEQBranch} \\
\sum_{p:a_{lp}=1, a_{mp}=1} X_p \leq 1 \label{CGModel:LEQBranch}
\end{align}

Rather than explicitly adding these constraints to the RMP, they are enforced by setting the infeasible columns' upper bound to zero. On the branch of~\eqref{CGModel:GEQBranch}, this is equivalent to combine items $l$ and $m$. For the branch of~\eqref{CGModel:LEQBranch}, this is equivalent to making rows $l$ and $m$ disjoint. However, this is not enough, since we must also avoid infeasible columns to be re-generated on each branch, which can be achieved by modifying the underlining pricing subproblem as follows: 
\begin{itemize}
  \item
    For the branches in which items $l$ and $m$ are combined, it suffices to modify the set of items to $\mathcal{\bar I} = \{1,\dots,n\} \setminus \{l, m\} \cup \{n+1\}$, with $\alpha_{n+1} = \alpha_l + \alpha_m$,
\begin{align}
  s_{n+1}^{k} = 
  \begin{cases}
    s_l^{k}, &\text{ if }  s_l^{k} > 0 \text{ and } s_m^{k} = 0, \\
    s_m^{k}, &\text{ if }  s_l^{k} = 0 \text{ and } s_m^{k} > 0, \\
    s_l^{k} + s_m^{k}, &\text{ if }  s_l^{k} > 0 \text{ and } s_m^{k} > 0, \\
    0, &\text{ otherwise;}
  \end{cases}
\end{align}

\item
  For branches in which items $l$ and $m$ must be packed in different patterns, we simply add the following disjunction constraint 
\begin{equation}
  A_l + A_m \leq 1.
\end{equation}
\end{itemize}
These changes can be easily accommodated into the model~\eqref{PricingModel:1}-\eqref{PricingModel:3} without much effort.

To obtain an upper bound at each node, we solve the current restricted master problem with the integrality constraint. In other words, we solve an ILP model that consists only of the columns that have been already generated. The use of this method at each node can be computationally expensive, but it usually provides tight upper bounds.

Some instances of the problem may be too complex to be solved by the column generation algorithm in practical time. In this manner, we would need faster methods for obtaining lower bounds for such instances. Two methods are proposed to get a lower bound at the root node of the branch-and-price algorithm. One of them is based on the continuous lower bound for the bin packing problem, and it is given by the following equation:

\begin{equation}
\label{LB:continuous}
\textit{LB}_{\textit{CON}} = \max_{k\in \mathcal{K}} \left\{ \ceil*{ \sum_{i\in S_k} s_{i}^k/W } \right\}.
\end{equation}

The idea behind the continuous lower bound~\eqref{LB:continuous} is to break items into unitary-sized items to fill empty spaces in a packing, which could not be filled otherwise. The $LB_{CON}$ computes the continuous lower bound for every scenario, and returns the largest value between them.

The other lower bound that we use is based on dual feasible functions. A \emph{dual feasible function} (DFF) is a function $f$ such that, for any finite set $X$ of real numbers, if $\sum_{x\in X}x \leq 1$, then $\sum_{x\in X}f(x) \leq 1$. For a survey on DFFs, we refer to \citet{Clautiaux10}. Given a DFF~$f$, we can construct a lower bound based on it, similarly to the continuous lower bound, as given in the following equation:

\begin{equation}
\label{LB:dff}
\textit{LB}_{\textit{DFF}} = \max_{k\in K} \{ \lceil \sum_{i\in S_k} f(s_{i}^k/W) \rceil \}
\end{equation}

In the proposed branch-and-price algorithm, given $\lambda \in \{1,\ldots, W/2\} $, we consider the DFF proposed by \citet{FeketeS01} and described in~\eqref{dff:fk}:

\begin{equation}
\label{dff:fk}
f(s) =
\begin{cases}
W, \text{ if }s > W-\lambda;\\
0, \text{ if }s \leq \lambda;\\
s, \text{ otherwise.}
\end{cases}
\end{equation}

\section{Numerical Experiments\label{sec-experiments}}
The efficiency of the proposed VNS and B\&P algorithm is evaluated with computational experiments. Both solutions methods were implemented in the C++ programming language, with the B\&P algorithm using the Gurobi Solver version 8.1~\cite{gurobi} to solve the linear programming models. The experiments were carried out in a computer with an Intel(R) Xeon(R) CPU E5--2630 v4 processor at 2.20 GHz, with 64 GB of RAM and under the Ubuntu 18.04 operating system.

A total of 120 instances were generated, divided into 12 different classes of~10 instances each. The classes are combinations of the number of items $n \in \{10, 50, 100, 200\}$ and the number of scenarios $d \in \{ 0.5n, n, 2n \}$, which depends on the number of items of the class. For every instance, the size of the items is randomly generated  over the set $\{ 1, \ldots, 99\}$ under a uniform probability distribution, whereas the size of the bin is fixed to $100$. The scenarios were generated randomly, similarly to~\citet{bodisB18}, such that an item $i\in \mathcal{I}$ belongs to the scenario $k\in \mathcal{K}$ with probability $0.5$.

A summary of the computational results is presented in Table \ref{tab:results}. Columns~$n$ and $d$ present, respectively, the number of items and the number of scenarios of each class of instances. We consider three sets of columns, each representing one of the tested algorithms, with the last one, VNS+B\&P, being the B\&P algorithm using the solution obtained from the VNS as a warm-start. The columns $gap$ represent the average optimality gap that each algorithm achieved in each class, computed as $\frac{UB-LB}{UB}$, where $UB$~and~$LB$ are respectively the upper and lower bounds obtained. Notice that since the VNS does not produce a lower bound on its own, we used the lower bound from the VNS+B\&P to compute its optimality gap. Columns $time$ represent the average time in seconds that each algorithm spent to solve each instance of the corresponding class, being $1800$ seconds the time limit for the VNS and VNS + B\&P algorithms, and a time limit of $3600$ for B\&P. Columns $|OPT|$ represent the number of instances from each class that were solved to optimality (out of 10) by each algorithm. Columns $cols$ and $nodes$ represent the average number of columns generated and the number of explored nodes, respectively, considering the B\&P algorithm.


\begin{landscape}
\clearpage
\begin{table}
\centering
\label{tab:results}
\begin{tabular}{@{}cccccccccccccccccc@{}}
\toprule
             &              & \multicolumn{1}{l}{} & \multicolumn{3}{c}{VNS}                          &  & \multicolumn{5}{c}{B\&P}                              &  & \multicolumn{5}{c}{VNS + B\&P}                        \\ \cmidrule(lr){4-6} \cmidrule(lr){8-12} \cmidrule(l){14-18} 
$n$          & $d$          & \multicolumn{1}{l}{} & $gap$ & $time$ & $|\text{OPT}|$ &  & $gap$ & $nodes$ & $cols$ & $time$  & $|\text{OPT}|$ &  & $gap$ & $nodes$ & $cols$  & $time$  & $|\text{OPT}|$ \\ \midrule
10           & 5            &                      & 0.00\%                & 0.00    & 10             &  & 0.00\%  & 1.0     & 6.9    & 0.02    & 10             &  & 0.00\%  & 0      & 0       & 0.00    & 10             \\
10           & 10           &                      & 0.00\%                & 0.00    & 10             &  & 0.00\%  & 1.6     & 14.0   & 0.03    & 10             &  & 0.00\%  & 0.3    & 5.4     & 0.00    & 10             \\
10           & 20           &                      & 1.67\%                & 0.00    & 9              &  & 0.00\%  & 1.2     & 11.5   & 0.03    & 10             &  & 0.00\%  & 0.3    & 3.4     & 0.00    & 10             \\
50           & 25           &                      & 7.11\%                & 14.80   & 1              &  & 0.00\%  & 6.4     & 156.1  & 11.21   & 10             &  & 1.01\%  & 955.8  & 98256.4 & 364.99  & 8              \\
50           & 50           &                      & 5.24\%                & 28.80   & 1              &  & 0.00\%  & 15.7    & 289.8  & 30.08   & 10             &  & 0.50\%  & 321.8  & 34501.3 & 187.00  & 9              \\
50           & 100          &                      & 7.26\%                & 28.80   & 0              &  & 0.00\%  & 270.2   & 2401.1 & 569.41  & 10             &  & 0.93\%  & 129.8  & 10871.7 & 370.18  & 8              \\
100          & 50           &                      & 9.34\%                & 134.30  & 0              &  & 1.44\%  & 80.8    & 2381.6 & 1832.76 & 5              &  & 0.56\%  & 6.6    & 4575.5  & 516.59  & 8              \\
100          & 100          &                      & 8.25\%                & 181.20  & 0              &  & 2.21\%  & 85.3    & 2399.6 & 3019.09 & 2              &  & 1.38\%  & 29.2   & 12419.1 & 1193.98 & 5              \\
100          & 200          &                      & 7.78\%                & 240.40  & 0              &  & 1.51\%  & 41.4    & 1028.1 & 2710.84 & 4              &  & 1.72\%  & 5.8    & 3577.8  & 1052.04 & 6              \\
200          & 100          &                      & 16.32\%               & 1048.80 & 0              &  & 13.74\% & 1.0     & 532.7  & 3814.57 & 0              &  & 20.13\% & 1      & 3323.1  & 1817.98 & 0              \\
200          & 200          &                      & 18.56\%               & 1407.70 & 0              &  & 17.52\% & 1.0     & 460.3  & 3925.27 & 0              &  & 20.07\% & 1      & 4801.9  & 1828.71 & 0              \\
200          & 400          &                      & 18.31\%               & 1595.30 & 0              &  & 20.16\% & 1.0     & 362.7  & 3866.15 & 0              &  & 20.29\% & 1      & 4832.4  & 1817.93 & 0              \\ \midrule
\multicolumn{2}{c}{Average} &                      & 8.32\%                & 390.01  & 2.58           &  & 4.71\%  & 42.21   & 837.0  & 1648.28 & 5.91           &  & 5.55\%  & 121.05 & 14764   & 762.45  & 6.16           \\ \bottomrule
\end{tabular}
\caption{Summary of the computational results for all the algorithms tested.}
\end{table}
\end{landscape}
\clearpage

First, we discuss the performance of the VNS algorithm. From Table~\ref{tab:results}, we can notice that most instances with $10$ items are easily solved by the heuristic. However, for instances of size $50$ and $100$ we can notice a significant larger gap, even though the algorithm time never reaches its time limit. This indicates that the VNS quickly found a local-optima and was not able to escape it, thus prematurely converging. Something similar can be observed for the instances of size $200$, although the running time gets closer to the time limit in these cases. For this algorithm, only two instances of size $50, 100$, and $200$ are solved to optimality, achieving a total average gap of $8.32$\% in an average time of $390.01$ seconds.  

Observing the results for algorithm B\&P, we can notice a significant improvement. All instances of $10$ and $50$ items are solved to optimality very quickly (100 seconds on average). We start to notice a drop in performance for the instances of size $100$ and $200$. For instances of size $100$, the algorithm was able to solve $11$ instances to optimality, out of $30$, averaging over $1800$ seconds for all number of scenarios. We can also observe that the average gap for these instances is very small (around $2$\% for all number of scenarios). The algorithm times out for all instances of size $200$, with an average gap of 14\%. Interestingly, the average gaps achieved by the B\&P for these instances are not far from the ones obtained by the VNS, with the latter having half the time limit of the former. The total average gap for this algorithm is $4.71$\% in an average time of $1648.28$ seconds.

Finally, we analyze the VNS + B\&P algorithm, that is, B\&P using the VNS solution as a warm-start. Once again, all instances of size $10$ are completely solved in negligible time. Differing from the pure B\&P algorithm though, this one could not find optimal solutions for a few instances of size $50$. These $5$ instances brought all the averages for this group of instances up, since they were tried until the time limit of $1800$ seconds, totaling thousands of columns and hundreds of nodes more than their non-warm-started counterpart. This might happen because the solution of the VNS algorithm might impact the branching step of the B\&P, making it harder for the algorithm to converge depending on the branch it takes initially. However, we can notice that for instances of size $100$, the VNS + B\&P solved $19$ instances to optimality, $8$ more than the pure B\&P and with a significant smaller average time. This algorithm also timed out for all instances of size $200$, while presenting a greater gap in average compared to the pure B\&P because of the smaller time limit. When compared to the pure B\&P algorithm, the VNS + B\&P gives an average gap greater by $0.16$\% while using less than half the time in average.

The data from Table~\ref{tab:results} shows us that although the VNS algorithm struggles to avoid local-optima, its results serve as good upper bounds to be fed into a more elaborated model such as the B\&P presented. By using the VNS's result as a warm-start, the B\&P algorithm was able to solve more instances in less time, this was not enough to tackle instances with $200$ items.

\section{Concluding Remarks\label{sec-conclusions}}
This work deals with a variant of the bin packing problem where items occupy the bin capacity when they belong to the same scenario. For this problem, a variety of solution methods are proposed, including an absolute approximation algorithm, an asymptotic polynomial time approximation scheme, a variable neighborhood search based heuristic, and a branch-and-price that solves a set-cover-based formulation. Results of the heuristic and the branch-and-price illustrate how effective is the latter to solve small and medium size instances, having its efficiency improved when the heuristic provides a start solutions. 

While the approximation algorithms depend either on an approximation algorithm for the vector bin packing problem, for the absolute approximation algorithm, or the number of scenarios, for the APTAS, the heuristic and branch-and-price can be used to handle the problem in practical contexts. The heuristic is not so competitive with the branch-and-price algorithm, since the former only optimally solved 26\% of the instances, while the latter obtained optimal solutions for 59\% of the instances. If solutions of the heuristic are used as a warm-start for the branch-and-price, the number of optimal solutions increases to about~62\%. In total, we were able to prove optimality for $79$ out of the $120$ proposed instances. 

Future research can focus on new approximation algorithms that explore the influence of the scenarios. Improvements in the variable neighborhood search are also expected, including new neighborhood structures and local searches. Regarding the branch-and-price, we will work on valid cuts and the possibility of dynamic programming algorithms for the pricing problem.

\section*{Compliance with Ethical Standards}
The authors acknowledge the financial support of the 
National Counsel of Technological and Scientific Development (CNPq grants 
144257/2019--0, 
312186/2020--7, 
311185/2020--7, 
311039/2020--0, 
405369/2021--2, 
313146/2022--5
), 
the State of Goi\'as Research Foundation (FAPEG), 
and the State of S\~ao Paulo Research Foundation (FAPESP grant 
2017/11831--1
). 
This study was financed in part by the Coordena\c{c}\~ao de Aperfei\c{c}oamento de Pessoal de N\'ivel Superior~-~Brasil (CAPES)~-~Finance Code 001.

\section*{Ethical approval}
This article does not contain any studies with human participants or animals performed by any of the authors.

\section*{Data availability}
The datasets generated during and/or analysed during the current study are available from the corresponding author on reasonable request.

\bibliographystyle{apa}
\bibliography{ref}

\begin{thebibliography}{}

\bibitem[\protect\astroncite{Alves et~al.}{2014}]{ALVES2014}
Alves, C., {Valério de Carvalho}, J., Clautiaux, F., and Rietz, J. (2014).
\newblock Multidimensional dual-feasible functions and fast lower bounds for
  the vector packing problem.
\newblock {\em European Journal of Operational Research}, 233(1):43--63.

\bibitem[\protect\astroncite{Baghalian et~al.}{2013}]{baghalianRF13}
Baghalian, A., Rezapour, S., and Farahani, R.~Z. (2013).
\newblock Robust supply chain network design with service level against
  disruptions and demand uncertainties: A real-life case.
\newblock {\em European journal of operational research}, 227(1):199--215.

\bibitem[\protect\astroncite{Baldi et~al.}{2019}]{BALDI2019}
Baldi, M.~M., Manerba, D., Perboli, G., and Tadei, R. (2019).
\newblock A generalized bin packing problem for parcel delivery in last-mile
  logistics.
\newblock {\em European Journal of Operational Research}, 274(3):990--999.

\bibitem[\protect\astroncite{Balogh et~al.}{2019}]{BALOGH19}
Balogh, J., Békési, J., Dósa, G., Sgall, J., and van Stee, R. (2019).
\newblock The optimal absolute ratio for online bin packing.
\newblock {\em Journal of Computer and System Sciences}, 102:1--17.

\bibitem[\protect\astroncite{Bansal et~al.}{2009}]{Bansal2009}
Bansal, N., Caprara, A., and Sviridenko, M. (2009).
\newblock A new approximation method for set covering problems, with
  applications to multidimensional bin packing.
\newblock {\em SIAM Journal on Computing}, 39(4):1256--1278.

\bibitem[\protect\astroncite{Beltr{\'a}n et~al.}{2004}]{beltranCCM04}
Beltr{\'a}n, J.~D., Calder{\'o}n, J.~E., Cabrera, R.~J., Moreno-P{\'e}rez,
  J.~A., and Moreno-Vega, J.~M. (2004).
\newblock Grasp-vns hybrid for the strip packing problem.
\newblock {\em Hybrid metaheuristics}, 2004:79--90.

\bibitem[\protect\astroncite{Bertazzi et~al.}{2019}]{BERTAZZI2019}
Bertazzi, L., Golden, B., and Wang, X. (2019).
\newblock The bin packing problem with item fragmentation:a worst-case
  analysis.
\newblock {\em Discrete Applied Mathematics}, 261:63--77.

\bibitem[\protect\astroncite{Bettinelli et~al.}{2010}]{bettinelliCR10}
Bettinelli, A., Ceselli, A., and Righini, G. (2010).
\newblock A branch-and-price algorithm for the variable size bin packing
  problem with minimum filling constraint.
\newblock {\em Annals of Operations Research}, 179(1):221--241.

\bibitem[\protect\astroncite{B{\'o}dis and Balogh}{2018}]{bodisB18}
B{\'o}dis, A. and Balogh, J. (2018).
\newblock Bin packing problem with scenarios.
\newblock {\em Central European Journal of Operations Research}, pages 1--19.

\bibitem[\protect\astroncite{Borges et~al.}{2020}]{borgesMSX20}
Borges, Y.~G., Miyazawa, F.~K., Schouery, R.~C., and Xavier, E.~C. (2020).
\newblock Exact algorithms for class-constrained packing problems.
\newblock {\em Computers \& Industrial Engineering}, 144:106455.

\bibitem[\protect\astroncite{Brandão and Pedroso}{2016}]{BRANDAO2016}
Brandão, F. and Pedroso, J.~P. (2016).
\newblock Bin packing and related problems: General arc-flow formulation with
  graph compression.
\newblock {\em Computers \& Operations Research}, 69:56--67.

\bibitem[\protect\astroncite{Buljubašić and Vasquez}{2016}]{BULJUBASIC2016}
Buljubašić, M. and Vasquez, M. (2016).
\newblock Consistent neighborhood search for one-dimensional bin packing and
  two-dimensional vector packing.
\newblock {\em Computers \& Operations Research}, 76:12--21.

\bibitem[\protect\astroncite{Caprara and Toth}{2001}]{CAPRARA2001}
Caprara, A. and Toth, P. (2001).
\newblock Lower bounds and algorithms for the 2-dimensional vector packing
  problem.
\newblock {\em Discrete Applied Mathematics}, 111(3):231--262.

\bibitem[\protect\astroncite{Chekuri and Khanna}{2004}]{chekuriK04}
Chekuri, C. and Khanna, S. (2004).
\newblock On multidimensional packing problems.
\newblock {\em SIAM journal on computing}, 33(4):837--851.

\bibitem[\protect\astroncite{Christensen et~al.}{2017}]{Christensen2017}
Christensen, H.~I., Khan, A., Pokutta, S., and Tetali, P. (2017).
\newblock Approximation and online algorithms for multidimensional bin packing:
  A survey.
\newblock {\em Computer Science Review}, 24:63--79.

\bibitem[\protect\astroncite{Clautiaux et~al.}{2010}]{Clautiaux10}
Clautiaux, F., Alves, C., and {Val{\'e}rio de Carvalho}, J. (2010).
\newblock A survey of dual-feasible and superadditive functions.
\newblock {\em Annals of Operations Research}, 179(1):317--342.

\bibitem[\protect\astroncite{Clautiaux et~al.}{2014}]{CLAUTIAUX2014}
Clautiaux, F., Dell'Amico, M., Iori, M., and Khanafer, A. (2014).
\newblock Lower and upper bounds for the bin packing problem with fragile
  objects.
\newblock {\em Discrete Applied Mathematics}, 163:73--86.

\bibitem[\protect\astroncite{Coffman et~al.}{2013}]{coffmanCGMV13}
Coffman, E.~G., Csirik, J., Galambos, G., Martello, S., and Vigo, D. (2013).
\newblock Bin packing approximation algorithms: survey and classification.
\newblock In {\em Handbook of combinatorial optimization}, pages 455--531.
  Springer.

\bibitem[\protect\astroncite{Dahmani et~al.}{2013}]{DAHMANI2013}
Dahmani, N., Clautiaux, F., Krichen, S., and Talbi, E.-G. (2013).
\newblock Iterative approaches for solving a multi-objective 2-dimensional
  vector packing problem.
\newblock {\em Computers \& Industrial Engineering}, 66(1):158--170.

\bibitem[\protect\astroncite{Dantzig and Wolfe}{1961}]{DW61}
Dantzig, G. and Wolfe, P. (1961).
\newblock The decomposition algorithm for linear programs.
\newblock {\em Econometrica}, 29(4):767--778.

\bibitem[\protect\astroncite{de~Lima et~al.}{2023}]{Lima2023}
de~Lima, V.~L., Iori, M., and Miyazawa, F.~K. (2023).
\newblock Exact solution of network flow models with strong relaxations.
\newblock {\em Mathematical Programming}, 197(2):813--846.

\bibitem[\protect\astroncite{Dell'Amico et~al.}{2020}]{DELLAMICO2020}
Dell'Amico, M., Furini, F., and Iori, M. (2020).
\newblock A branch-and-price algorithm for the temporal bin packing problem.
\newblock {\em Computers \& Operations Research}, 114:104825.

\bibitem[\protect\astroncite{Dell’Amico et~al.}{2020}]{dellamicoFI20}
Dell’Amico, M., Furini, F., and Iori, M. (2020).
\newblock A branch-and-price algorithm for the temporal bin packing problem.
\newblock {\em Computers \& Operations Research}, 114:104825.

\bibitem[\protect\astroncite{Delorme and Iori}{2019}]{Delorme19}
Delorme, M. and Iori, M. (2019).
\newblock Enhanced pseudo-polynomial formulations for bin packing and cutting
  stock problems.
\newblock {\em INFORMS Journal on Computing}, 32(1):101--119.

\bibitem[\protect\astroncite{Delorme et~al.}{2016a}]{DELORME2016}
Delorme, M., Iori, M., and Martello, S. (2016a).
\newblock Bin packing and cutting stock problems: Mathematical models and exact
  algorithms.
\newblock {\em European Journal of Operational Research}, 255(1):1--20.

\bibitem[\protect\astroncite{Delorme et~al.}{2016b}]{DelormeI16}
Delorme, M., Iori, M., and Martello, S. (2016b).
\newblock Bin packing and cutting stock problems: Mathematical models and exact
  algorithms.
\newblock {\em European Journal of Operational Research}, 255(1):1--20.

\bibitem[\protect\astroncite{Fekete and Schepers}{2001}]{FeketeS01}
Fekete, S.~P. and Schepers, J. (2001).
\newblock New classes of fast lower bounds for bin packing problems.
\newblock {\em Mathematical Programming}, 91(1):11--31.

\bibitem[\protect\astroncite{{Fernandez de la Vega} and Lueker}{1981}]{VegaL81}
{Fernandez de la Vega}, W. and Lueker, G.~S. (1981).
\newblock Bin packing can be solved within $1+\epsilon $ in linear time.
\newblock {\em Combinatorica}, 1(4):349--355.

\bibitem[\protect\astroncite{Feuerstein et~al.}{2014}]{feuersteinMSSSSZ14}
Feuerstein, E., Marchetti-Spaccamela, A., Schalekamp, F., Sitters, R., van~der
  Ster, S., Stougie, L., and van Zuylen, A. (2014).
\newblock Scheduling over scenarios on two machines.
\newblock In {\em International Computing and Combinatorics Conference}, pages
  559--571.

\bibitem[\protect\astroncite{Fleszar and Hindi}{2002}]{fleszarH02}
Fleszar, K. and Hindi, K.~S. (2002).
\newblock New heuristics for one-dimensional bin-packing.
\newblock {\em Computers \& operations research}, 29(7):821--839.

\bibitem[\protect\astroncite{Ford and Fulkerson}{1958}]{FF58}
Ford, L. and Fulkerson, D. (1958).
\newblock A suggested computation for maximal multi-commodity network flows.
\newblock {\em Management Science}, 5(1):97--101.

\bibitem[\protect\astroncite{Garey and Johnson}{1978}]{gareyJ78}
Garey, M.~R. and Johnson, D.~S. (1978).
\newblock ''strong'' np-completeness results: Motivation, examples, and
  implications.
\newblock {\em Journal of the ACM (JACM)}, 25(3):499--508.

\bibitem[\protect\astroncite{Gilmore and Gomory}{1961}]{GilmoreG61}
Gilmore, P. and Gomory, R. (1961).
\newblock A linear programming approach to the cutting stock problem.
\newblock {\em Operations Research}, 9:849--859.

\bibitem[\protect\astroncite{Gilmore and Gomory}{1963}]{GilmoreG63}
Gilmore, P. and Gomory, R. (1963).
\newblock A linear programming approach to the cutting stock problem - part
  {II}.
\newblock {\em Operations Research}, 11:863--888.

\bibitem[\protect\astroncite{Grange et~al.}{2018}]{GRANGE2018}
Grange, A., Kacem, I., and Martin, S. (2018).
\newblock Algorithms for the bin packing problem with overlapping items.
\newblock {\em Computers \& Industrial Engineering}, 115:331--341.

\bibitem[\protect\astroncite{Gurobi~Optimization}{2023}]{gurobi}
Gurobi~Optimization, L. (2023).
\newblock Gurobi optimizer reference manual.

\bibitem[\protect\astroncite{Heßler et~al.}{2018}]{HELER2018}
Heßler, K., Gschwind, T., and Irnich, S. (2018).
\newblock Stabilized branch-and-price algorithms for vector packing problems.
\newblock {\em European Journal of Operational Research}, 271(2):401--419.

\bibitem[\protect\astroncite{Hu et~al.}{2017}]{HU2017}
Hu, Q., Zhu, W., Qin, H., and Lim, A. (2017).
\newblock A branch-and-price algorithm for the two-dimensional vector packing
  problem with piecewise linear cost function.
\newblock {\em European Journal of Operational Research}, 260(1):70--80.

\bibitem[\protect\astroncite{Juan et~al.}{2018}]{juanKCF18}
Juan, A.~A., Kelton, W.~D., Currie, C.~S., and Faulin, J. (2018).
\newblock Simheuristics applications: dealing with uncertainty in logistics,
  transportation, and other supply chain areas.
\newblock In {\em 2018 winter simulation conference (WSC)}, pages 3048--3059.

\bibitem[\protect\astroncite{Kantorovich}{1960}]{Kantorovich60}
Kantorovich, L.~V. (1960).
\newblock Mathematical methods of organizing and planning prodution.
\newblock {\em Management Science}, 6:363--422.
\newblock (in Russian 1939).

\bibitem[\protect\astroncite{Kucukyilmaz and Kiziloz}{2018}]{KUCUKYILMAZ2018}
Kucukyilmaz, T. and Kiziloz, H.~E. (2018).
\newblock Cooperative parallel grouping genetic algorithm for the
  one-dimensional bin packing problem.
\newblock {\em Computers \& Industrial Engineering}, 125:157--170.

\bibitem[\protect\astroncite{Mladenovi{\'c} and Hansen}{1997}]{mladenovicH97}
Mladenovi{\'c}, N. and Hansen, P. (1997).
\newblock Variable neighborhood search.
\newblock {\em Computers \& operations research}, 24(11):1097--1100.

\bibitem[\protect\astroncite{Pereira}{2016}]{PEREIRA2016}
Pereira, J. (2016).
\newblock Procedures for the bin packing problem with precedence constraints.
\newblock {\em European Journal of Operational Research}, 250(3):794--806.

\bibitem[\protect\astroncite{Puchinger and Raidl}{2008}]{puchingerR08}
Puchinger, J. and Raidl, G.~R. (2008).
\newblock Bringing order into the neighborhoods: relaxation guided variable
  neighborhood search.
\newblock {\em Journal of Heuristics}, 14(5):457--472.

\bibitem[\protect\astroncite{Ryan and Foster}{1981}]{ryanF81}
Ryan, D.~M. and Foster, B.~A. (1981).
\newblock An integer programming approach to scheduling.
\newblock {\em Computer scheduling of public transport urban passenger vehicle
  and crew scheduling}, pages 269--280.

\bibitem[\protect\astroncite{Sadykov and Vanderbeck}{2013}]{sadykovV13}
Sadykov, R. and Vanderbeck, F. (2013).
\newblock Bin packing with conflicts: a generic branch-and-price algorithm.
\newblock {\em INFORMS Journal on Computing}, 25(2):244--255.

\bibitem[\protect\astroncite{Saint-Guillain et~al.}{2021}]{guillainPL21}
Saint-Guillain, M., Paquay, C., and Limbourg, S. (2021).
\newblock Time-dependent stochastic vehicle routing problem with random
  requests: Application to online police patrol management in brussels.
\newblock {\em European Journal of Operational Research}, 292(3):869--885.

\bibitem[\protect\astroncite{Sweeney and Paternoster}{1992}]{sweeneyP92}
Sweeney, P.~E. and Paternoster, E.~R. (1992).
\newblock Cutting and packing problems: a categorized, application-orientated
  research bibliography.
\newblock {\em Journal of the Operational Research Society}, 43(7):691--706.

\bibitem[\protect\astroncite{Vance et~al.}{1994}]{vanceBHN94}
Vance, P.~H., Barnhart, C., Johnson, E.~L., and Nemhauser, G.~L. (1994).
\newblock Solving binary cutting stock problems by column generation and
  branch-and-bound.
\newblock {\em Computational optimization and applications}, 3(2):111--130.

\bibitem[\protect\astroncite{Vazirani}{2001}]{Vazirani01}
Vazirani, V. (2001).
\newblock {\em Approximation Algorithms}.
\newblock Springer-Verlag.

\bibitem[\protect\astroncite{W{\"a}scher et~al.}{2007}]{WascherHS07}
W{\"a}scher, G., Haussner, H., and Schumann, H. (2007).
\newblock An improved typology of cutting and packing problems.
\newblock {\em European Journal of Operational Research}, 183(3):1109--1130.

\bibitem[\protect\astroncite{Wei et~al.}{2020a}]{WEI2020}
Wei, L., Lai, M., Lim, A., and Hu, Q. (2020a).
\newblock A branch-and-price algorithm for the two-dimensional vector packing
  problem.
\newblock {\em European Journal of Operational Research}, 281(1):25--35.

\bibitem[\protect\astroncite{Wei et~al.}{2020b}]{WLBL20}
Wei, L., Luo, Z., Baldacci, R., and Lim, A. (2020b).
\newblock A new branch-and-price-and-cut algorithm for one-dimensional
  bin-packing problems.
\newblock {\em {INFORMS} Journal on Computing}, 32(2):428--443.

\bibitem[\protect\astroncite{Woeginger}{1997}]{woeginger97}
Woeginger, G.~J. (1997).
\newblock There is no asymptotic ptas for two-dimensional vector packing.
\newblock {\em Information Processing Letters}, 64(6):293--297.

\bibitem[\protect\astroncite{Xu et~al.}{2016}]{xuLLJ16}
Xu, X.-y., Liu, J., Li, H.-y., and Jiang, M. (2016).
\newblock Capacity-oriented passenger flow control under uncertain demand:
  Algorithm development and real-world case study.
\newblock {\em Transportation Research Part E: Logistics and Transportation
  Review}, 87:130--148.

\end{thebibliography}


\end{document}